\documentclass[12pt]{article}
\usepackage{graphicx}
\usepackage[ruled,vlined]{algorithm2e}
\usepackage{appendix}
\usepackage{amsmath}
\usepackage{amssymb}
\usepackage{amsthm}
\usepackage{multirow}
\usepackage{color}
\usepackage{longtable}
\usepackage{array}
\usepackage{url}
\usepackage{booktabs}
\usepackage{float}
\usepackage{mathtools}
\usepackage{tikz}
\usepackage{caption}

\newtheorem{theorem}{Theorem}[section]

\newtheorem{lemma}[theorem]{Lemma}
\newtheorem{example}[theorem]{Example}

\newtheorem{corollary}[theorem]{Corollary}
\newtheorem{proposition}[theorem]{Proposition}

\theoremstyle{definition}
\newtheorem{remark}[theorem]{Remark}

\newcommand{\turan}{Tur\'{a}n }
\newcommand{\erdos}{Erd\H{o}s}

\title{On vertex-induced weighted \turan problems}

\begin{document}
\author{Zixiang Xu$^{\text{a,}}$\thanks{Email address: \texttt{zxxu8023@qq.com}},~~Yifan Jing$^{\text{b,}}$\thanks{ Email address: \texttt{yifanjing17@gmail.com}}~~and Gennian Ge$^{\text{a,}}$\thanks{Email address:~\texttt{gnge@zju.edu.cn.} Research supported by the National Natural Science Foundation of China under Grant Nos.~11431003 and 61571310, Beijing Scholars Program, Beijing Hundreds of Leading Talents Training Project of Science and Technology, and Beijing Municipal Natural Science Foundation.}\\
\footnotesize $^{\text{a}}$ School of Mathematical Sciences, Capital Normal University, Beijing, 100048, China\\
\footnotesize $^{\text{b}}$ Department of Mathematics, University of Illinois at Urbana Champaign, Urbana, Illinois, 61801, USA.\\}

\date{}

\maketitle

\begin{abstract}
Recently, Bennett et al.~introduced the vertex-induced weighted \turan problem.
In this paper, we consider their open \turan problem under sum-edge-weight function and characterize the extremal structure of $K_{l}$-free graphs.~We also prove the stability result of $K_{l}$-free graphs.~Based on these results, we propose a generalized version of the \erdos-Stone theorem for weighted graphs under two types of vertex-induced weight functions.
\end{abstract}


\smallskip


\section{Introduction}
One of the oldest results in extremal graph theory, which states that every graph on $n$ vertices with more than $n/4$ edges contains a triangle, was proved by Mantel \cite{1907Mantel} in $1907$. This result was generalized later to $K_{l}$-free graphs by \turan \cite{1941Turan}. Furthermore, the \erdos-Stone theorem is an asymptotic version generalizing \turan theorem to bound the number of edges in an $H$-free graph for a non-complete graph $H$. It is named after Paul \erdos \ and Arthur Stone, who proved it in $1946$ \cite{1946ErodsBAMS}, and it has been described as the ``fundamental theorem of extremal graph theory" \cite{1998BollobasGTM}. From then on, many papers were written on
proposing some extensions and generalizations of the \erdos-Stone theorem, see e.g. \cite{1996AlonJCTB, 1973BollobasG1, 1975BollobasG2, 1996BollobasJCTB, 1981CHevatalJLMS, 2015GeelenCombinatorica, 2002IshigamiEJC}.

\par A weighted graph $G$ is the one in which each edge $e$ is assigned a real number $w(e)$, called the weight
of $e$. Weighted graphs are widely studied due to their applications in the field of combinatorial optimization and artificial intelligence, where people are concerned with finding a subgraph of specified type of minimum or maximum weight in polynomial time. While in recent decades, there have been several topics on weighted graphs in the field of extremal graph theory. For example, one of the topics is to determine the number of edges needed in a weighted graph on $n$ vertices to guarantee that it contains a subgraph of specified type.    Bondy and Fan~\cite{1989BondyAnnDM, 1991BondyComb} began the research on such problems, and some classical theorems on the existence of  long paths and cycles in unweighted graphs were generalized to weighted graphs. A number of papers, such as \cite{2009JunDM, 2010LiBinglongDM, 2000ZhangDM}, have proposed some related generalizations of Bondy-Fan's results.

\par There were also many papers that proposed weighted version of some classical theorems. A weighted generalization of Ore's theorem was obtained by Bondy, Broersma, van den Heuvel and Veldman \cite{2002BondyCGT}, then Bollob\'{a}s and Scott \cite{1996BollobasJCTB} examined similar questions for weighted digraphs. Mathew and Sunitha \cite{2011MathewAML} generalized one of the celebrated results in graph theory due to Menger, which plays a crucial role in many areas of flow and network theory. Bondy and Tuza \cite{1997BondyJGT} gave such a generalization of Tur\'an's theorem, and an integer-weighted version of the \turan  problem has been investigated which can
be viewed as a multi-graph version of the original \erdos' problem. In \cite{1997BondyJGT}, they also characterized the extremal graphs in certain cases, and additional cases were obtained by F\"{u}redi \cite{2002FurediJGT} and Kuchenbrod in his thesis \cite{1999KuchenbrodPHD}.
\par Recently Bennett et al.~\cite{2018BennettProduct} proposed a new type of weighted \turan problem where the weight of edges are induced by incident endpoints. The aim of the question is to determine the maximum weight of $G$ without some given $H$ and to show the structure of the extremal graphs. They considered two types of vertex-induced functions such as product-edge-weight and min-edge-weight functions. Under these two functions, they characterized the structures of the extremal $K_{l}$-free graphs. Using these results, they solved some multipartite \turan problem and calculated the maximum rectilinear crossing numbers of certain trees. They also posed an open \turan problem under sum-edge-weight function.

\par In this paper, firstly we focus on sum-edge-weight \turan problems which were posed by Bennett et al. We show that if all vertex weights are strictly positive, the extremal $K_{l}$-free graph $F$ must be a complete $(l-1)$-partite graph, and we show the difference between the structures of extremal graphs on sum-edge-weight and product-edge-weight functions. We also show that for every $K_{l+1}$-free weighted graph, we can remove a set of edges with bounded total weight to make the graph $l$-partite. Using the Szemer\'{e}di regularity lemma and counting lemma, we further propose the weighted version of the famous \erdos-Stone theorem on both sum-edge-weight and product-edge-weight functions. We show that for an arbitrary finite graph $H$ with $\chi(H)\geqslant 3$, the strategy to find the extremal $H$-free graph with maximum weight is to search from the complete $(\chi(H)-1)$-partite graphs, and the error term is $o(n^{2})$.

\par The rest of this paper is organized as follows. In Section \ref{sec:SEWTuran}, we introduce some definitions and basic facts on the sum-edge-weight \turan problems. In Section \ref{sec:Klfree}, we first prove our results on the triangle-free case with specified arguments, and then we generalize them to the $K_{l}$-free case by induction. We then show the different strategies to find the extremal graphs on sum-edge-weight and product-edge-weight functions. We also prove the stability result of $K_{l}$-free graphs. The weighted versions of \erdos-Stone theorem are proposed in Section \ref{sec:weightedESS}, with the standard argument by Szemer\'{e}di regularity lemma and counting lemma. We conclude in Section \ref{sec:Conclusion}.

\section{Sum-edge-weight \turan Problems}\label{sec:SEWTuran}

Let $w$: $V(K_{n})\mapsto [0,\infty)$ be a vertex weight function, and the sum-edge-weight $w_{+}$: $E(K_{n})\mapsto [0,\infty)$ be given by $w_{+}(uv)=w(u)+w(v).$ We define the sum-edge-weight extremal number as
\begin{equation*}
  ex(n,w_{+},H):=\max\{w_{+}(G)|G\subseteq K_{n},G\text{ is }H\text{-free}\},
\end{equation*}
 where $w_{+}(G):=\sum\limits_{e\in E(G)}w_{+}(e)$, and $F$ is called \emph{extremal graph} if
 \begin{equation*}
   w_{+}(F)=\max\{w_{+}(G)|G\subseteq K_{n},G\text{ is } H\text{-free}\}.
 \end{equation*}

 It is clear that 
\begin{equation*}
  w_{+}(G)=\sum_{i=1}^{n}d(v_{i})w(v_{i}),
\end{equation*}
where $d(v)$ is the degree of vertex $v$. By using this formula, we obtain the following proposition about the degree distribution of vertices in extremal graphs.
\begin{proposition}\label{pro:degree}
  Suppose $F$ is the extremal graph and $u,v\in F$. If $w(u)\geqslant w(v)$, then we have $d(u)\geqslant d(v).$
\end{proposition}
\begin{proof}
Suppose there exist two vertices $u,v\in F$ such that $w(u)\geqslant w(v)$ and $d(u) < d(v).$ We will exchange the weight of vertex $u$ and vertex $v$, and then obtain a new graph $F'$. Note that this operation does not change the structure of $F$. Since
$$w(u)d(v)+w(v)d(u)> w(u)d(u)+w(v)d(v),$$ we find that $w_{+}(F')> w_{+}(F),$ which is a contradiction.
\end{proof}

Proposition \ref{pro:degree} shows that if we want to maximize $w_{+}(G)$, the vertices with larger weight should be adjacent to more vertices.

\section{Forbidding a complete graph}\label{sec:Klfree}

\subsection{Case: Triangle-free}
We first consider $ex(n,w_{+},K_{3}).$ We will prove that if all vertex weights are strictly positive, then any extremal triangle-free graph must be complete bipartite.

\begin{lemma}\label{lem:K3}
  The sum-edge-weight extremal number $ex(n,w_{+},K_{3})$ is equal to
  \begin{equation*}
    \max_{r}\{(n-r)\sum_{i=1}^{r}w(v_{i})+r\sum_{j=r+1}^{n}w(v_{j})\}.
  \end{equation*}
\end{lemma}

\begin{proof}
  Here we prove a stronger result that if all vertex weights are strictly positive, then any extremal triangle-free graph must be the complete bipartite graph. Suppose that $F$ is the extremal graph such that $F$ is non-bipartite, let the degree sequence of $F$ be $(d(v_{1}),d(v_{2}),\ldots,d(v_{n}))$ and without loss of generality, we assume that $w(v_{1})\geqslant w(v_{2})\geqslant \cdots\geqslant w(v_{n})$, by Proposition \ref{pro:degree}, we have $d(v_{1})\geqslant d(v_{2})\geqslant \cdots \geqslant d(v_{n}).$ Then we construct a new bipartite graph $F'$ with left hand $L$ and right hand $R$, where $R$ is the set of neighborhood of $v_{1}$, that is, $|L|=n-d(v_{1})$, $|R|=d(v_{1})$. Then
  \begin{equation*}
    w_{+}(F')=d(v_{1})\sum\limits_{v_{i}\in L}w(v_{i})+(n-d(v_{1}))\sum\limits_{v_{i}\in R}w(v_{i}),
  \end{equation*}
  and
  \begin{equation*}
    w_{+}(F')-w_{+}(F)=\sum\limits_{v_{i}\in L}(d(v_{1})-d(v_{i}))w(v_{i})+\sum\limits_{v_{i}\in R}(n-d(v_{1})-d(v_{i}))w(v_{i}).
  \end{equation*}

Observe that each vertex $u_{1}$ in $R$ cannot be adjacent to any other vertex $u_{2}$ in $R$, otherwise $\{u_{1},u_{2},v_{1}\}$ forms a triangle. This observation tells us that if $v_{i}\in R$, then $d(v_{i})\leqslant n-d(v_{1}).$ Through the above analysis, since all vertex weights are strictly positive and $F$ is non-bipartite, we obtain that $w_{+}(F')-w_{+}(F)>0$, which is a contradiction.  Hence if $H$ is the extremal graph, then $F$ is bipartite, finishing the proof.
\end{proof}

\begin{remark}
  In the proof of Lemma \ref{lem:K3}, we just construct a new graph $F'$ such that the degree of every vertex in $F'$ is larger than or equal to that in $F$. Note that $F'$ does not have to be the extremal graph. By Proposition \ref{pro:degree}, to construct the extremal graph $F=L\cup R$, we should put vertices $v_{1},v_{2},\ldots,v_{r}$ into the left hand $L$ and put the remaining vertices into the right hand $R$, where the value of $r$ depends on the vertex weight function.
\end{remark}

\subsection{Case: $K_{l}$-free}
In the previous subsection, we have proved that for given vertex function, if all of vertex weights are strictly positive and the graph does not contain any copies of triangle, then the extremal graph must be complete bipartite. It is natural to ask whether there are similar results in $K_{l}$-free graphs, where $K_{l}$ is a clique of $l$ vertices. We will give the positive answer.

\begin{lemma}\label{lem:Kl}
  For any $K_{l}$-free graph $G$ with $n$ vertices, there exists a complete $(l-1)$-partite graph $\tilde{G}$ with the following property: For every $1\leqslant i \leqslant n$, the degree of $v_{i}$ in $\tilde{G}$ is at least the degree of $v_{i}$ in $G$, with equalities to hold for every $i$ when $G$ is a complete $(l-1)$-partite graph.
\end{lemma}

\begin{proof}

  We proceed via induction on $l$. Our basic case is $l=3$, which has been proved in Lemma \ref{lem:K3}. Suppose the conclusion holds for $l-1$, then we consider the $K_{l}$-free graph $G=(V,E)$. For ease, denote $d_{i}:=d(v_{i})$ as the degree of $v_{i}$ in $G$, and without loss of generality, we assume that $w(v_{1})\geqslant w(v_{2})\geqslant \cdots\geqslant w(v_{n})$. By Proposition \ref{pro:degree}, we have $d_{1}\geqslant d_{2}\geqslant\cdots\geqslant d_{n}$. Let $N(v_{1})$ be the set of neighbors of $v_{1}$ in $G$, and $G[N(v_{1})]$ be the induced subgraph of $G$ with vertex set $N(v_{1})$. Then $|N(v_{1})|=d_{1}$, and $G[N(v_{1})]$ is $K_{l-1}$-free, otherwise there would be a copy of $K_{l}$ in $G$. So by the induction hypothesis, there exists a complete $(l-2)$-partite graph $M$, with vertex set $N(v_{1})$, such that for any $v\in N(v_{1})$, the degree of $v$ in $M$ is larger than or equal to the degree of $v$ in $G[N(v_{1})]$, i.e., $\forall v_{i}\in N(v_{1}),$
   \begin{equation*}
      d_{M}(v_{i})\geqslant d_{G[N(v_{1})]}(v_{i}),
   \end{equation*}
   with equality to hold only when $G[N(v_{1})]$ is a complete $(l-2)$-partite graph. Then we construct a new graph $\tilde{G}=(V,\tilde{E})$, the edge set consists of two parts:
  \begin{equation*}
    \tilde{E}=\{(v,u)|v\in V\setminus N(v_{1}),\ u\in N(v_{1})\}\cup E(M).
  \end{equation*}
  For ease, denote $U=V\setminus(\{v_{1}\}\cup N(v_{1}))$, we then compare $\tilde{G}=(V,\tilde{E})$ with $G=(V,E)$ as follows.
  \begin{enumerate}
    \item $d_{\tilde{G}}(v_{1})=d_{G}(v_{1})=d_{1}.$
    \item $\forall v_{i}\in U$, $d_{\tilde{G}}(v_{i})=d_{\tilde{G}}(v_{1})=d_{1}\geqslant d_{G}(v_{i}).$
    \item $\forall v_{i}\in N(v_{1})$, $d_{\tilde{G}}(v_{i})=d_{M}(v_{i})+|U|+1 \geqslant d_{G[N(v_{1})]}(v_{i})+|U|+1 \geqslant d_{G}(v_{i}).$
  \end{enumerate}
  Hence $\tilde{G}=(V,\tilde{E})$ is a complete $(l-1)$-partite graph such that for every vertex $v_{i}$, the degree of $v_{i}$ in $\tilde{G}$ is larger than or equal to the degree of $v_{i}$ in $G$, with equalities to hold for all $1\leqslant i\leqslant n$ when $G[N(v_{1})]$ is a complete $(l-2)$-partite graph, and all vertices in $U$ are adjacent to all vertices in $N(v_{1})$. Note that this is equivalent to saying that $G$ is a complete $(l-1)$-partite graph, finishing the proof.
\end{proof}

Using Lemma \ref{lem:Kl}, we state that if all of vertex weights are strictly positive and the graph does not contain any copies of $K_{l}$ with $l\geqslant 3$, then the extremal graph must be a complete $(l-1)$-partite graph.
\begin{theorem}\label{thm:Kl}
  The sum-edge-weight extremal number $ex(n,w_{+},K_{l})$ is equal to
  \begin{equation*}
    \max\limits_{\mathcal{P}}\sum\limits_{P\in\mathcal{P}}(n-|P|)w(P),
  \end{equation*}
  where the maximum is taken over all partitions $\mathcal{P}$ of $V(K_{n})$ into $l-1$ parts, $|P|$ is the number of vertices in $P$, and $w(P)=\sum\limits_{v\in P}w(v)$.
\end{theorem}

\begin{proof}
  Suppose $G=(V,E)$ is the extremal graph with no copies of $K_{l},$ and $G$ is not a complete $(l-1)$-partite graph, by Lemma \ref{lem:Kl}, we can always find a complete $(l-1)$-partite graph $\tilde{G}=(V,\tilde{E})$ such that for every vertex $v_{i}\in V$,
  \begin{equation*}
    d_{\tilde{G}}(v_{i})\geqslant d_{G}(v_{i}),
  \end{equation*}
 and the equalities cannot hold for all vertices, hence we obtain that
 \begin{equation*}
   w_{+}(\tilde{G})=\sum_{i=1}^{n}d_{\tilde{G}}(v_{i})w(v_{i})> w_{+}(G)=\sum_{i=1}^{n}d_{G}(v_{i})w(v_{i}),
 \end{equation*}
 which is a contradiction to the assumption that $G$ is the extremal graph.
\end{proof}

Note that $\tilde{G}$ does not have to be the extremal graph, and how to partition $n$ vertices into $l-1$ parts depends on the vertex weight function.

Now we review and compare our results with the similar results in \cite{2018BennettProduct},  where the authors showed that under product-weight-edge function, the extremal $K_{l}$-free graph must be a complete $(l-1)$-partite graph.

\begin{proposition}\textup{\cite{2018BennettProduct}}\label{prop:ProductKl}
  The extremal number $ex(n,w_{\pi},K_{l})$ is equal to
  \begin{equation*}
    \max_{\mathcal{P}}\sum_{\substack{ P,P'\in\mathcal{P} \\ P\neq P'}}w(P)w(P'),
  \end{equation*}
  where $w(P)=\sum\limits_{v\in P}w(v)$ and the maximum is taken over all partitions $\mathcal{P}$ of $V(K_{n})$ into $l-1$ parts.
\end{proposition}

It seems that the results of both cases are the same, that is, both of the extremal $K_{l}$-free graphs are complete $(l-1)$-partite graphs. It is natural to ask whether the structures of these two graphs are exactly the same. In the following we show a simple example and give a negative answer to this question.

\begin{example}\label{exam:K3}
  We set $n=6$, and let $w(v_{1},v_{2},v_{3},v_{4},v_{5},v_{6})=(41,33,29,13,11,7)$. The extremal triangle-free graphs of two vertex-induced functions are shown in Figure \ref{fig:sum} and Figure \ref{fig:product} respectively.

  \begin{minipage}{0.4\textwidth}
\begin{center}
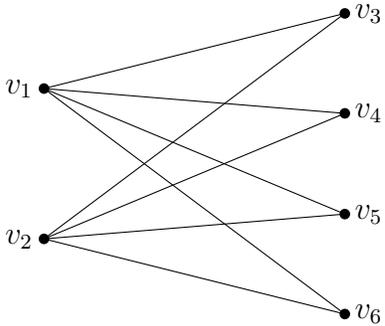

\begin{tikzpicture}[trim left, trim right=2cm, baseline]
\fill (-2,1) circle (2pt);
\node (node1)[left] at (-2,1)  {$v_{1}$};
\fill (-2,-1) circle (2pt);
\node (node2)[left] at (-2,-1) {$v_{2}$};
\fill (2,2) circle (2pt);
\node (node3)[left] at (2.65,2)  {$v_{3}$};
\fill (2,0.67) circle (2pt);
\node (node4)[left] at (2.65,0.67)  {$v_{4}$};
\fill (2,-0.67) circle (2pt);
\node (node5)[left] at (2.65,-0.67) {$v_{5}$};
\fill (2,-2) circle (2pt);
\node (node5)[left] at (2.65,-2) {$v_{6}$};

\draw (-2,1) -- (2,2);
\draw (-2,1) -- (2,0.67);
\draw (-2,1) -- (2,-0.67);
\draw (-2,1) -- (2,-2);
\draw (-2,-1) -- (2,2);
\draw (-2,-1) -- (2,0.67);
\draw (-2,-1) -- (2,-0.67);
\draw (-2,-1) -- (2,-2);
\end{tikzpicture}
\captionsetup{font=footnotesize}
\captionof{figure}{Sum-edge-weight}\label{fig:sum}
\end{center}
\end{minipage}
\begin{minipage}{0.6\textwidth}
\begin{center}
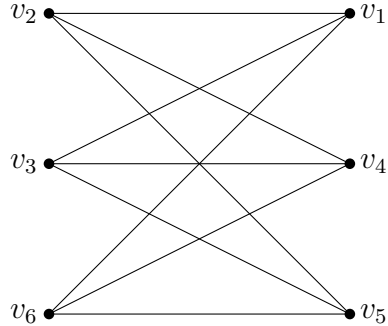

\begin{tikzpicture}[trim left, trim right=2cm, baseline]
\fill (-2,2) circle (2pt);
\node (node1)[left] at (-2,2)  {$v_{2}$};
\fill (-2,0) circle (2pt);
\node (node2)[left] at (-2,0) {$v_{3}$};
\fill (-2,-2) circle (2pt);
\node (node3)[left] at (-2,-2)  {$v_{6}$};
\fill (2,2) circle (2pt);
\node (node4)[left] at (2.65,2)  {$v_{1}$};
\fill (2,0) circle (2pt);
\node (node5)[left] at (2.65,0) {$v_{4}$};
\fill (2,-2) circle (2pt);
\node (node5)[left] at (2.65,-2) {$v_{5}$};

\draw (-2,2) -- (2,2);
\draw (-2,2) -- (2,0);
\draw (-2,2) -- (2,-2);
\draw (-2,0) -- (2,2);
\draw (-2,0) -- (2,0);
\draw (-2,0) -- (2,-2);
\draw (-2,-2) -- (2,2);
\draw (-2,-2) -- (2,0);
\draw (-2,-2) -- (2,-2);
\end{tikzpicture}
\captionsetup{font=footnotesize}
\captionof{figure}{Product-edge-weight}\label{fig:product}
\end{center}
\end{minipage}

\end{example}

\begin{remark}
  In sum-edge-weight case, Proposition \ref{pro:degree} shows that the vertices with larger weight should be adjacent to more vertices.  As we can see in Figure \ref{fig:sum}, vertices with larger weight are put in the left part which contains fewer vertices than those in the right part. While in the product-edge-weight case, for given vertex weight function, the sum of weight of all vertices is fixed, then by Cauchy-Schwarz inequality, the strategy is to minimize the weight difference between each part. In Example \ref{exam:K3}, the total weight of $6$ vertices is $134$, and the minimum value of the difference is $4$, then we obtain the extremal graph as shown in Figure \ref{fig:product}.
\end{remark}

The following theorem tells us that for every $K_{l+1}$-free graph $G$, we can remove some edges with bounded total weight to make the graph $G$ $l$-partite.

\begin{theorem}
  Let $G$ be a $K_{l+1}$-free graph, such that $w_{+}(G)=ex(n,w_{+},K_{l+1})-t$. Then we can make the graph $G$ $l$-partite by removing a set of edges $E_{t}$ with $w_{+}(E_{t})\leqslant t$.
\end{theorem}

\begin{proof}
By Proposition~\ref{pro:degree}, we may assume $d(v)\leqslant d(u)$ if $w(v)\leqslant w(u)$ in $G$. Otherwise we can map $G$ to $G'$ such that $G'$ is isomorphic to $G$ and satisfies the above degree conditions, we still have $ex(n,w_+,K_{l+1})-w_{+}(G')\leqslant t$.

Let $v_{1},v_{2},\ldots,v_{n}$ be an arbitrary linear order of $V(G)$, and for every subset $X\subseteq V(G)$, denote $d_{X}(v)=|N_{X}(v)|$ as the number of neighbors of $v$ in $X$. We now consider the following algorithm on $G$.

In the first step, let $X_{1}:=V(G)$ and we pick $v_{1}\in X_{1}$ such that $d_{X_{1}}(v_{1})=\max\limits_{u\in V(G)} d_{X_{1}}(u)$ and $v_1$ has the smallest index in the linear order. Let $V_{1}=X_{1}\setminus N(v_{1})$. In step $i$, we pick $v_{i}\in X_{i}:=V(G)\setminus (\bigcup\limits_{j=1}^{i-1}V_{j})$, such that $d_{X_{i}}(v_{i})=\max\limits_{u\in V(G)} d_{X_{i}}(u)$, and $v_{i}$ has the smallest index, and we let $V_{i}=X_i\setminus N_{X_i}(v_i)$. The algorithm stops in step $p$ when $X_{p+1}=\emptyset$. Since $G$ is $K_{l+1}$-free, we have $p\leqslant l$.

For step $i$ of the algorithm, we have
\begin{align*}
   & 2\sum_{e\in E(V_{i})}w_{+}(e)+\sum_{e\in E(V_{i}\times X_{i+1})}w_{+}(e) \\
   =& \sum_{v\in V_{i}}d_{X_{i}}(v)w(v)+\sum_{u\in X_{i+1}}d_{V_{i}}(u)w(u) \\
\leqslant& |X_{i+1}|w(V_{i})+|V_{i}|w(X_{i+1}),
\end{align*}
where $w(X)=\sum\limits_{x\in X}w(x)$. Let $H_{p}$ be a complete $p$-partite graph, with vertex partition $V_{1},V_{2},\dots,V_{p}$. Then we have
\begin{align*}
   &\sum_{e\in E(G)}w_{+}(e)+\sum_{i=1}^p \sum_{e\in E(V_i)} w_{+}(e) \\
   =&\sum_{i=1}^p\Big(2\sum_{e\in E(V_i)}w_{+}(e)+\sum_{e\in E(V_i\times X_{i+1})}w_{+}(e)\big)\\
   \leqslant &\sum_{i=1}^p \Big(|V_i|w(X_{i+1})+|X_{i+1}|w(V_i)\Big)\\
   =&\sum_{e\in E(H_{p})}w_{+}(e)\leqslant ex(n,w_+,K_{l+1}),
   \end{align*}
which gives us $\sum\limits_{i=1}^{p} \sum\limits_{e\in E(V_i)} w_{+}(e)\leqslant t$. Since we can remove the edges set $E_{t}=\bigcup\limits_{i=1}^{p}E(V_{i})$ to make the graph $G$ $l$-partite, the proof is finished.
\end{proof}

\section{Weighted \erdos-Stone type theorem}\label{sec:weightedESS}
In Theorem \ref{thm:Kl}, we consider the problem of determining the maximum weight and the structure of extremal graph without a clique of given size. We have proved that for given vertex function, if all of the vertex weights are strictly
positive and the graph does not contain any copies of $K_{l}$, then the extremal graph must be complete
$(l-1)$-partite. Now we will generalize this result by forbidding an arbitrary subgraph, our result can be viewed as a weighted version of \erdos-Stone theorem.\\

We recall \erdos-Stone theorem and let $\chi(H)$ denote the chromatic number of a graph $H$, the minimum number of colors needed to color $V(H)$ properly.

\begin{theorem}[\erdos-Stone Theorem]\label{thm:ESS}
  Fix a graph $H$, as $n \rightarrow \infty$, we have
  \begin{equation*}
    ex(n,H)=\bigg(1-\frac{1}{\chi(H)-1}+o(1)\bigg)\binom{n}{2}.
  \end{equation*}
\end{theorem}

While in the weighted version, if all of vertex weights are strictly positive, we will obtain an asymptotic result of $ex(n,w_{+},H)$ as follows.

\begin{theorem}\label{thm:weightedESS}
  Fix a graph $H$, as $n \rightarrow \infty$, we have
  \begin{equation*}
    ex(n,w_{+},H)=\max\limits_{\mathcal{R}}\sum\limits_{P\in\mathcal{R}}(n-|P|)w(P)+o(n^{2}),
  \end{equation*}
  where the maximum is taken over all the partitions $\mathcal{R}$ of $V(K_{n})$ into $\chi(H)-1$ parts.
\end{theorem}

Before we prove Theorem \ref{thm:weightedESS}, we introduce the famous Szemer\'{e}di regularity lemma.

\begin{lemma}[Szemer\'edi Regularity Lemma]\label{lem:psrl}
  For all $m,\epsilon >0$, there exists an integer $M$ such that the following holds. If a graph $G$ has $n\geqslant m$ vertices, there exists an integer $k$ satisfying $m \leqslant k \leqslant M$, and a partition $\mathcal{P}$ of $V(G)$ with $k+1$ parts $V_{0},V_{1},\ldots,V_{k},$ for every $1\leqslant i<j\leqslant k+1,$ $|V_{i}|=|V_{j}|$, and $|V_0|<\epsilon n$. For all but at most $\epsilon k^{2}$ pairs $(V_{i},V_{j})$ ($1\leqslant i<j\leqslant k+1$), for every subsets $A\subseteq V_{i}$ and $B\subseteq V_{j}$ with $|A|\geqslant \epsilon |V_{i}|$ and $|B|\geqslant \epsilon |V_{j}|,$ we have
  \begin{equation*}
   \left|\frac{e(A,B)}{|A|\cdot|B|}-\frac{e(V_{i},V_{j})}{|V_{i}|\cdot|V_{j}|}\right| < \epsilon,
  \end{equation*}
  where $e(A,B)$ is the number of edges having one end vertex in $A$ and one end vertex in $B$.
\end{lemma}

   Let $X$, $Y$ be disjoint subsets of $V$, we define the density of the pair $(X,Y)$ as
   \begin{equation*}
     d(X,Y):=\frac{e(X,Y)}{|X|\cdot |Y|},
   \end{equation*}
   and we call a pair of vertex sets $X$ and $Y$ are \emph{$\epsilon$-regular}, if for all subsets $A\subseteq X$, $B\subseteq Y$ satisfying $|A|\geqslant \epsilon|X|$, $|B|\geqslant \epsilon|Y|$, we have
   \begin{equation*}
   \left|\frac{e(A,B)}{|A|\cdot|B|}-\frac{e(V_{i},V_{j})}{|V_{i}|\cdot|V_{j}|}\right|=\left|d(X,Y)-d(A,B)\right| \leqslant \epsilon.
\end{equation*}
   A partition of $V$ into $k+1$ sets $(V_0,V_{1},\ldots,V_{k})$ is called an \emph{$\epsilon$-regular} partition, if
   \begin{enumerate}
     \item $|V_0|<\epsilon n$ and for all $1\leqslant i<j\leqslant k+1$ we have $|V_{i}|=|V_{j}|$.
     \item all except $\epsilon k^{2}$ of the pairs $V_{i},V_{j}$, $1\leqslant i<j\leqslant k$, are $\epsilon$-regular.
   \end{enumerate}

Now we introduce an application of Szemer\'{e}di regularity lemma, which was proposed in \cite{2012DiestelGT}. Let $G$ be a graph with an $\epsilon$-regular partition $\{V_{0},V_{1},\ldots,V_{k}\}$, with exceptional set $V_{0}$ and $|V_{1}|=|V_{2}|=\dots=|V_{k}|=l$. Given $d\in[0,1],$ let $R$ be the graph on $\{V_{1},V_{2},\ldots,V_{k}\}$ in which two vertices $V_{i}$ and $V_{j}$ are adjacent in $R$ if and only if they form an $\epsilon$-regular pair in $G$ of density $\geqslant$ $d$. We shall call $R$ a regularity graph of $G$ with parameters $\epsilon$, $l$ and $d$. Given $s\in \mathbb{N}$, replace every vertex $V_{i}$ of $R$ by a set $V_{i}^{s}$ of $s$ vertices, and every edge by a complete bipartite graph between the corresponding $s$-sets, the resulting graph will be denoted as $R_{s},$ and we usually call $R_{s}$ an $s$-blow up graph of $R$. The following lemma says that subgraphs of $R_{s}$ can also be found in $G$ under some conditions.

\begin{lemma}\label{lemma:Counting}\textup{\cite{2012DiestelGT}}
 For all $d\in(0,1]$ and $\Delta \geqslant 1$ there exists an $\epsilon_{0}>0$ with the following property: if $G$ is any graph, $H$ is a graph with $\Delta(H)\leqslant \Delta$, $s\in \mathbb{N}$, and $R$ is any regularity graph with parameters $\epsilon\leqslant\epsilon_{0}$, $l\geqslant \max\{\frac{2s}{d^{\Delta}},d\}$, then we have
 \begin{equation*}
   H\subseteq R_{s} \ \Rightarrow H\subseteq G.
 \end{equation*}
\end{lemma}

 With all tools in hand, we are going to prove Theorem \ref{thm:weightedESS}.
 \begin{proof}[Proof of Theorem \ref{thm:weightedESS}]
   For arbitrary real number $t>0$, let $\epsilon=\frac{t}{9T}$, where
   \begin{equation*}
     T=\max\limits_{e\in E(K_{n})}w_{+}(e).
   \end{equation*}
   Suppose $G$ is a subgraph of $K_{n}$ and $G$ is $H$-free. We apply Szemer\'{e}di regularity lemma on $G$ and $\epsilon$, and with $m\geqslant 1/\epsilon$. We obtain an $\epsilon$-partition $\{V_{0},V_{1},\ldots, V_{k}\}$, with exceptional set $V_{0}$ and $|V_{1}|=|V_{2}|=\cdots=|V_{k}|=l$. Then let $R$ be the regularity graph of $G$, that is, $|V(R)|=k$ and for every $1\leqslant i<j\leqslant k$, $V_{i}$ and $V_{j}$ are adjacent in $R$ if and only if $V_{i}$ and $V_{j}$ are $\epsilon$-regular in $G$ and with edge density at least $10\epsilon$. We define a function $f$: $V(R)\mapsto [0,\infty)$ where $f(V_{i})=\frac{\sum_{v\in V_{i}}w(v)}{\left|V_{i}\right|}.$ Let $F$ be a complete graph based on $V(R)$, that is, $\left|V(F)\right|=k$ and it has the same vertex weight as $R$. \\

   Since $H$ is a finite graph with chromatic number $\chi(H)$ and $\Delta(H)\leqslant\Delta,$ $H$ is a subset of some complete $\chi(H)$-partite graph. Now we claim $R$ is $K_{\chi(H)}$-free, otherwise there exists some constant $s$, such that $H\subseteq R_{s}$. It's easy to check that $l=\frac{n-|V_{0}|}{k}> \max \{\frac{2s}{(10\epsilon)^{2}},10\epsilon\}$, hence $H\subseteq G$ by Lemma \ref{lemma:Counting}, which is a contradiction. Then by the result of $K_{l}$-free graphs in Theorem \ref{thm:Kl}, we have
   \begin{equation*}\label{eq:1}
     \sum\limits_{e\in E(R)}w_{+}(e)\leqslant \max\limits_{\mathcal{P}}\sum\limits_{P\in \mathcal{P}}(k-\left|P\right|)f(P),
   \end{equation*}
   where the maximum is taken over all partitions $\mathcal{P}$ of $V(F)$ into $\chi(H)-1$ parts.
   Let $\hat{G}$ be the $|V_{i}|$-blow up graph of $R$, that is, we construct $\hat{G}$ by replacing every vertex of $R$ with an independent set of $|V_{i}|$ vertices, and every edge with a complete bipartite graph between the corresponding independent sets. $\hat{G}$ has the same vertex weight as $G$, we have

   \begin{flalign*}
     \sum_{e\in E(\hat{G})}w_{+}(e)= &\left|V_{i}\right|^{2}\max\limits_{\mathcal{P}}\sum\limits_{P\in \mathcal{P}}(k-\left|P\right|)f(P) \\
     = & \max\limits_{\mathcal{P}}\sum\limits_{P\in\mathcal{P}}(n-\left|P\right|\cdot\left|V_{i}\right|)\sum\limits_{V_{i}\in \mathcal{P}}w(V_{i})  \\
     \leqslant & \max\limits_{R}\sum_{P\in \mathcal{R}}(n-\left|P\right|)w(P),
   \end{flalign*}
     where the maximum is taken over all partitions $\mathcal{R}$ of $V(K_{n})$ into $\chi(H)-1$ parts. It's clear that
     \begin{equation*}
       \sum\limits_{1\leqslant i<j\leqslant k}\sum\limits_{e\in E'}w_{+}(e)\leqslant \sum\limits_{e\in E(\hat{G})}w_{+}(e),
     \end{equation*}
     where $E'\subseteq E(G)$ contains all edges between $V_i,V_j$ ($1\leqslant i,j\leqslant k$) in $G$, for every $\epsilon$-regular pair $(V_{i},V_{j})$ with density at least $10\epsilon$.

      Through the above analysis, we can classify the edges in $E(G)$ but not in $E(\hat{G})$ as follows.
     \begin{enumerate}
       \item Edges $(x,y)\in V_{i}\times V_{j}$ where $(V_{i},V_{j})$ ($1\leqslant i,j\leqslant k$) are not $\epsilon$-regular.
       \item Edges $(x,y)\in V_{i}\times V_{j}$ where the density of $(V_{i},V_{j})$ is less than $10\epsilon$.
       \item $x\in V_i$, where $0\leqslant i\leqslant k$.
       \item Edges $(x,y)\in V_{0}\times V_i$,  where $1\leqslant i\leqslant k$.
     \end{enumerate}

     Since $T$ is the maximum edge weight in weighted graph $V(K_{n})$, then we can estimate the total weight of the edges in $E(G)$ but not in $E(\hat{G})$.

     \begin{equation*}
       \sum\limits_{e\in E(G)\setminus E(\hat{G})}w_{+}(e)< \Big[\epsilon k^{2}\Big(\frac{n}{k}\Big)^{2}+10\epsilon\binom{k}{2}\Big(\frac{n}{k}\Big)^{2}+k\Big(\frac{n}{k}\Big)^2+\binom{\epsilon n}{2} +\epsilon n\Big(\frac{n}{k}\Big)k\Big]T <9\epsilon Tn^{2}=tn^{2}.
     \end{equation*}
   Through the above analysis, we obtain that
   \begin{equation*}
     \sum\limits_{e\in E(G)}w_{+}(e)\leqslant \max\limits_{\mathcal{R}}\sum\limits_{P\in \mathcal{R}}(n-\left|P\right|)w(P)+tn^{2}.
   \end{equation*}
  Since we take $t$ arbitrarily, the proof of Theorem \ref{thm:weightedESS} is finished.
 \end{proof}

 Now we consider the product-edge-weight \turan problem, which was introduced by Bennett et al.~\cite{2018BennettProduct}. Let $w:V(G)\to[0,\infty)$ be a vertex weight function and $w_{\pi}: E(G)\to [0,\infty)$ be the product-edge-weight function such that $w_{\pi}(uv)=w(u)w(v)$. The following result is implied by Proposition \ref{prop:ProductKl} and Theorem~\ref{thm:weightedESS}, we omit further details.

\begin{corollary}
  Fix a graph $H$, as $n \rightarrow \infty$, we have
  \begin{equation*}
    ex(n,w_{\pi},H)=\max\limits_{\mathcal{R}}\sum\limits_{
    \begin{subarray}{c}
   P,P'\in\mathcal{R}\\
   P\neq P'
   \end{subarray}}
   w(P')w(P)+o(n^{2}),
  \end{equation*}
  where $w(P)=\sum\limits_{v\in P}w(v)$ and the maximum is taken over all the partitions $\mathcal{R}$ of $V(K_{n})$ into $\chi(H)-1$ parts.
\end{corollary}

\section{Conclusion}\label{sec:Conclusion}
In this paper, we consider the vertex-induced weighted \turan problems introduced by Bennett et al.~\cite{2018BennettProduct}. Our contributions are mainly focused on the sum-edge-weight \turan problems. We first characterize the extremal structure of $K_{l}$-free graph under sum-edge-weight function and give the similarities and differences with the previous structure under product-edge-weight function. We also show the stability result of $K_{l}$-free graphs. Based on the results of complete graphs, we further propose a generalized version of the \erdos-Stone theorem for weighted graphs under both of the two types of vertex-induced weight functions. In future research, we hope to find more applications that can take advantage of these results.

\section{Acknowledgement}\label{sec:Acknowledgement}
The first author thanks his teammates Chengfei Xie for his observation, and Wenjun Yu, Xiangliang Kong for their helpful comments, he also thanks Professor Baogang Xu for the  guidance during his short visit at Capital Normal University.

\bibliographystyle{plain}
\bibliography{TuranXzxref}

\begin{thebibliography}{10}

\bibitem{1996AlonJCTB}
N.~Alon and R.~Yuster.
\newblock {$H$}-factors in dense graphs.
\newblock {\em J. Combin. Theory Ser. B}, 66(2):269--282, 1996.

\bibitem{2018BennettProduct}
P.~Bennett, S.~English, and M.~Talanda-Fisher.
\newblock Weighted tu\'{r}an problems with applications.
\newblock {\em arXiv preprint}, arXiv: 1809.05028, 2018.

\bibitem{1998BollobasGTM}
B.~Bollob\'{a}s.
\newblock {\em Modern graph theory}, volume 184 of {\em Graduate Texts in
  Mathematics}.
\newblock Springer-Verlag, New York, 1998.

\bibitem{1973BollobasG1}
B.~Bollob\'{a}s and P.~Erd\H{o}s.
\newblock On the structure of edge graphs.
\newblock {\em Bull. London Math. Soc.}, 5:317--321, 1973.

\bibitem{1975BollobasG2}
B.~Bollob\'{a}s, P.~Erd\H{o}s, and M.~Simonovits.
\newblock On the structure of edge graphs. {II}.
\newblock {\em J. London Math. Soc. (2)}, 12(2):219--224, 1975/76.

\bibitem{1996BollobasJCTB}
B.~Bollob\'{a}s and A.~D. Scott.
\newblock A proof of a conjecture of {B}ondy concerning paths in weighted
  digraphs.
\newblock {\em J. Combin. Theory Ser. B}, 66(2):283--292, 1996.

\bibitem{2002BondyCGT}
J.~A. Bondy, H.~Broersma, J.~van~den Heuvel, and H.~Veldman.
\newblock Heavy cycles in weighted graphs.
\newblock {\em Discuss. Math. Graph Theory}, 22(1):7--15, 2002.
\newblock Conference on Graph Theory (Elgersburg, 2000).

\bibitem{1989BondyAnnDM}
J.~A. Bondy and G.~Fan.
\newblock Optimal paths and cycles in weighted graphs.
\newblock In {\em Graph theory in memory of {G}. {A}. {D}irac ({S}andbjerg,
  1985)}, volume~41 of {\em Ann. Discrete Math.}, pages 53--69. North-Holland,
  Amsterdam, 1989.

\bibitem{1991BondyComb}
J.~A. Bondy and G.~Fan.
\newblock Cycles in weighted graphs.
\newblock {\em Combinatorica}, 11(3):191--205, 1991.

\bibitem{1997BondyJGT}
J.~A. Bondy and Z.~Tuza.
\newblock A weighted generalization of {T}ur\'{a}n's theorem.
\newblock {\em J. Graph Theory}, 25(4):267--275, 1997.

\bibitem{1981CHevatalJLMS}
V.~Chv\'{a}tal and E.~Szemer\'{e}di.
\newblock On the {E}rd{\H{o}}s-{S}tone theorem.
\newblock {\em J. London Math. Soc. (2)}, 23(2):207--214, 1981.

\bibitem{2012DiestelGT}
R.~Diestel.
\newblock {\em Graph Theory, 4th Edition}, volume 173 of {\em Graduate texts in
  mathematics}.
\newblock Springer, 2012.

\bibitem{1946ErodsBAMS}
P.~Erd\H{o}s and A.~H. Stone.
\newblock On the structure of linear graphs.
\newblock {\em Bull. Amer. Math. Soc.}, 52:1087--1091, 1946.

\bibitem{2009JunDM}
J.~Fujisawa.
\newblock Weighted degrees and heavy cycles in weighted graphs.
\newblock {\em Discrete Math.}, 309(23-24):6483--6495, 2009.

\bibitem{2002FurediJGT}
Z.~F\"{u}redi and A.~K\"{u}ndgen.
\newblock Tur\'{a}n problems for integer-weighted graphs.
\newblock {\em J. Graph Theory}, 40(4):195--225, 2002.

\bibitem{2015GeelenCombinatorica}
J.~Geelen and P.~Nelson.
\newblock An analogue of the {E}rd{\H{o}}s-{S}tone theorem for finite
  geometries.
\newblock {\em Combinatorica}, 35(2):209--214, 2015.

\bibitem{2002IshigamiEJC}
Y.~Ishigami.
\newblock A common extension of the {E}rd{\H{o}}s-{S}tone theorem and the
  {A}lon-{Y}uster theorem for unbounded graphs.
\newblock {\em European J. Combin.}, 23(4):431--448, 2002.

\bibitem{1999KuchenbrodPHD}
J.~Kuchenbrod.
\newblock {\em Extremal problems on weighted graphs}.
\newblock ProQuest LLC, Ann Arbor, MI, 1999.
\newblock Thesis (Ph.D.)--University of Kentucky.

\bibitem{2010LiBinglongDM}
B.~Li and S.~Zhang.
\newblock On extremal weighted digraphs with no heavy paths.
\newblock {\em Discrete Math.}, 310(10-11):1640--1644, 2010.

\bibitem{1907Mantel}
W.~Mantel.
\newblock Problem 28.
\newblock {\em Wiskundige Opgaven}, 10:60--61, 1907.

\bibitem{2011MathewAML}
S.~Mathew and M.~S. Sunitha.
\newblock A generalization of {M}enger's theorem.
\newblock {\em Appl. Math. Lett.}, 24(12):2059--2063, 2011.

\bibitem{1941Turan}
P.~Tur\'{a}n.
\newblock Eine extremalaufgabe aus der graphentheorie.
\newblock {\em Fiz Lapok}, pages 436--452, 1941.

\bibitem{2000ZhangDM}
S.~Zhang, X.~Li, and H.~Broersma.
\newblock Heavy paths and cycles in weighted graphs.
\newblock {\em Discrete Math.}, 223(1-3):327--336, 2000.

\end{thebibliography}
\end{document}